\documentclass[12pt,a4paper]{article}
\usepackage{amsfonts}
\usepackage{amssymb}
\usepackage{mathrsfs}
\usepackage{amsmath}
\usepackage{amsthm}
\usepackage{enumerate}
\usepackage{cite}
\allowdisplaybreaks
\newtheorem{lemma}{Lemma}[section]
\newtheorem{theorem}[lemma]{Theorem}
\newtheorem{remark}[lemma]{Remark}
\newtheorem{proposition}[lemma]{Proposition}
\newtheorem{definition}[lemma]{Definition}
\newtheorem{corollary}[lemma]{Corollary}
\newtheorem{example}[lemma]{Example}

\begin{document}

\title{\textbf{The structure of restricted Leibniz algebras}
\author{ Baoling Guan$^{1,2}$,  Liangyun Chen$^{1}$
 \date{{\small {$^1$ School of Mathematics and Statistics, Northeast Normal
 University,\\
Changchun 130024, China}\\{\small {$^2$ College of Sciences, Qiqihar
University, Qiqihar 161006, China}}}}}}

\date{ }
\maketitle
\begin{quotation}
\small\noindent \textbf{Abstract}: The paper studies the structure
of restricted Leibniz algebras. More specifically speaking, we first
give the equivalent definition of restricted Leibniz algebras, which
   is by far more tractable than that of a
restricted  Leibniz algebras in [6]. Second, we obtain some
properties of $p$-mappings and restrictable Leibniz algebras, and
discuss restricted Leibniz algebras with semisimple elements.
Finally, Cartan decomposition and the uniqueness of decomposition
for restricted Leibniz algebras are determined.

\noindent{\textbf{Keywords}}: Restricted Leibniz algebras; restrictable Leibniz algebras; $p$-mapping; semisimple elements; Tori; Cartan decomposition

 \small\noindent \textbf{Mathematics
Subject Classification 2000}: 17A32, 17B50
\renewcommand{\thefootnote}{\fnsymbol{footnote}}
 \footnote[0]{Corresponding
author(L. Chen): chenly640@nenu.edu.cn.}
 \footnote[0]{Supported by NNSF
of China (No.11171055),  Natural Science Foundation of Jilin
province (No. 201115006), Scientific Research Foundation for
Returned Scholars
    Ministry of Education of China and the Fundamental Research Funds for the Central Universities(No. 11SSXT146).}
\end{quotation}
\setcounter{section}{0}

\section{Introduction}

    The concept of a  restricted Lie algebra is
       attributable to N. Jacobson  in 1943. It is well known that  the Lie algebras
       associated with algebraic groups   over a field of  characteristic $p$  are
       restricted Lie algebras \cite{sf}. Now,
   restricted Lie algebras attract more and more attentions.
 For example:  restricted Lie superalgebras\cite{pe},  restricted
        color Lie  algebras\cite{ba},  restricted Lie triple system\cite{htl} and
     restricted Leibniz algebras \cite{dl} were studied, respectively. As is well
known,  restricted Lie  algebras play predominant roles in the
theories of modular Lie algebras \cite{sf1}. Analogously,  the study
of restricted Leibniz algebras  will play
      an important role in
    the  classification  of the  finite-dimensional modular simple Leibniz algebras.

      Leibniz algebras were first introduced
 as nonantisymmetric generalization of Lie algebras  in 1979 \cite{lj}. In
recent years the study of Leibniz algebras over a field of prime
characteristic  obtained  some important results. In \cite{das},
Dzhumadil'daev and Abdykassymova (2001) introduced the notion of
restricted Leibniz algebras(left Leibniz algebras). In \cite{dl},
the authors mainly proved that there is a functor-$p$-Leib from the
category of diassociative algebras to the category of restricted
Leibniz algebras(right Leibniz algebras) and constructed its
restricted enveloping algebra.
   As   a natural generalization of a restricted Lie  algebra, it seems desirable to investigate the possibility of establishing
 a parallel theory  for  restricted  Leibniz algebras.  However, in dealing
    with  a  restricted   Leibniz algebras,
    we can not employ all methods of restricted Lie algebras. This is because the product in  Leibniz algebras does not have skew symmetry.

 Similar to restricted Lie algebras, the paper gives the structure of
restricted Leibniz algebras(left Leibniz algebras). Let us briefly
describe the content and  setup of the present article. In Sec. 2,
the equivalent definition of restricted Leibniz algebras is given,
which
   is by far more tractable than that of a
restricted  Leibniz algebras in [6].
   In Sec. 3, we obtain some properties of $p$-mappings and restrictable Leibniz algebras.
  In Sec. 4, we study restricted Leibniz algebras whose elements are semisimple.  In Sec. 5, Tori and Cartan decomposition of restricted Leibniz algebras are
  characterized. In Sec. 6, the uniqueness of decomposition for restricted Leibniz algebras is determined.

 In the paper, $\mathbb{F}$ is a field of prime characteristic. Let $L$ denote a finite-dimensional Leibniz
 algebra(left Leibniz algebras) over $\mathbb{F}$.
We write $\mathbb{N}$ for nonnegative integers. For restricted
Leibniz algebra, the concepts of homomorphisms and
$p$-homomorphisms, derivations, $p$-representations are similar to
restricted Lie algebras\cite{sf}. $\mathrm{Der}L$ is also denoted by
the set consisting of all derivations of Leibniz algebra $L.$ All
other notions and concepts refer to the reference \cite{sf}.

\begin{definition} {\rm\cite{pa}}\,  A Leibniz algebra over $\mathbb{F}$ is an $\mathbb{F}$-module $L$
equipped with a bilinear mapping, called bracket,
              $$[-,-]:L\times L\rightarrow L$$
satisfying the Leibniz identity:
    $$[[x,y],z]=[x,[y,z]]-[y,[x,z]]$$
for all $ x,y,z\in L.$
\end{definition}

\begin{lemma}{\rm\cite{sf}} \label{l2.1.1}\,  Let $V$ and $W$ be $\mathbb{F}$-vector spaces and $f:V\rightarrow W$ be a $p$-semilinear mapping.
Then the following statements hold:

$\mathrm{(1)}$ $\mathrm{ker}(f)$ is an $\mathbb{F}$-subspace of $V.$

$\mathrm{(2)}$ $f(V)$ is an $\mathbb{F}^{p}$-subspace of $W.$ If $\mathbb{F}$ is perfect, then $f(V)$ is an  $\mathbb{F}$-subspace of $W.$

$\mathrm{(3)}$ $\mathrm{dim}_{\mathbb{F}}V=\mathrm{dim}_{\mathbb{F}}\mathrm{ker}(f)+\mathrm{dim}_{\mathbb{F}^{p}}f(V).$

$\mathrm{(4)}$ If $\langle f(V)\rangle=W$ and $\mathrm{dim}_{\mathbb{F}}W=\mathrm{dim}_{\mathbb{F}}V,$ then $\mathrm{ker}(f)=0.$
\end{lemma}

\begin{lemma} {\rm\cite{sf}}\label{l2.1.2}\,  Let $f:V\rightarrow V$  be $p$-semilinear. Then the following statements are equivalent.

$\mathrm{(1)}$ $\langle f(V)\rangle=V.$

$\mathrm{(2)}$ For every $v\in V,$ there exist $\alpha_{1},\cdots, \alpha_{n}\in \mathbb{F}$ such that $v=\sum\limits_{i=1}^{n}\alpha_{i}f^{i}(v).$
\end{lemma}

\begin{lemma}{\rm\cite{sf}} \label{l2.2.2}\,  Let $f$ be an endomorphism of a vector space $V$ and let $\chi$ be a polynomial such that $\chi(f)=0.$
 Then the following statements hold:

$\mathrm{(1)}$ If $\chi =q_{1}q_{2}$ and $q_{1}, q_{2}$ are relatively prime, then $V$ decomposes into a direct sum of $f$-invariant subspaces $V=U\oplus W$ such that $q_{1}(f)(U)=0=q_{2}(f)(W).$

$\mathrm{(2)}$ $V$ decomposes into a direct sum of $f$-invariant subspaces $V =V_{0}\oplus V_{1},$
for which $f|_{V_{0}}$ is nilpotent and $f|_{V_{1}}$ is invertible.

\end{lemma}

\begin{lemma} {\rm\cite{sf}}\label{l2.2.3}\,  Let $V$ be a vector space over  $\mathbb{F}$ and let $x,y$ be
elements of $\mathrm{End}_{\mathbb{F}}(V)$ such that there is $t\in
\mathbb{N}\backslash\{0\}$ with $ (\mathrm{ad}x)^{t}(y)=0.$ Suppose
that $q\in \mathbb{F}[\chi]$ is a polynomial, then $V_{0}(q(x))$ is
invariant under $y.$
\end{lemma}

\begin{definition} {\rm\cite{pa}}\, Let $H$ be a subspace of Leibniz algebra $L.$
$H$ is called a subalgebra of $L$, if $[H,H]\subseteq H$; $H$ is
called a left ideal of $L$, if $[L,H]\subseteq H;$ $H$ is called a
right ideal of $L$, if $[H,L]\subseteq H;$ $H$ is called an ideal of
$L$, if $[L,H]\subseteq H$ and $[H,L]\subseteq H.$
\end{definition}

\begin{definition} {\rm\cite{pa}} \, Let $L$ be a Leibniz algebra. The sequence $(L^{n})_{n\in\mathbb{N}\backslash \{0\}}$ of Leibniz algebra
$L$ given by $L^{1}:=L,$ $L^{n+1}:=[L,L^{n}].$ Then
$(L^{n})_{n\in\mathbb{N}\backslash \{0\}}$ is the descending central
series of $L.$ $L$ is called nilpotent, if there is $t\in
\mathbb{N}\backslash\{0\}$ such that $L^{t}=0.$  An abelian Leibniz
algebra $L$ is described by the condition $L^{2}=0.$
\end{definition}
\begin{definition} {\rm\cite{aao}}\, Let $L$ be a Leibniz algebra. The sequence $(L^{[n]})_{n\in\mathbb{N}\backslash \{0\}}$
defined by means of $L^{[1]}:=L,$ $L^{[n+1]}:=[L^{[n]},L^{[n]}]$ is
called the derived series of $L.$ $L$ is called solvable, if there
is $t\in \mathbb{N}\backslash\{0\}$ such that $L^{[t]}=0.$
\end{definition}

\begin{theorem} {\rm\cite{b}}\, (Engel's Theorem) Let $L$ be a Leibniz algebra. Suppose that the left multiplication operator $L_{a}$ is nilpotent for all $a\in L.$ Then $L$ is nilpotent.
\end{theorem}
\begin{definition}{\rm\cite{bdw}}\, A bimodule of Leibniz algebra $L$ is a vector space $M$ over $\mathbb{F}$
equipped with two bilinear compositions denoted by $ma$ and $am,$ for any $a\in L$ and $m\in M,$ satisfy
$$(ma)b=m[a,b]-a(mb),$$
$$(am)b=a(mb)-m[a,b],$$
$$[a,b]m=a(bm)-b(am).$$
\end{definition}
In \cite{bdw}, the author denotes by $\mathrm{End}(M)$ the associative algebra of all endomorphisms of the vector space $M.$
If $M$ is a bimodule of Leibniz algebra $L,$ then each of the mappings $S_{a}:m\rightarrow ma$
and $T_{a}:m\rightarrow am$ is an endmorphism of $M,$ and the mappings $S:a\rightarrow S_{a},$ $T:a\rightarrow T_{a}$
of $L$ into $\mathrm{End}(M)$ are linear. Moreover, $L_{[a,b]}=L_{a}L_{b}-L_{b}L_{a}$ for all $a,b\in L.$ Thus the set $\{L_{a}|a\in L\}$
forms a Lie algebra of linear transformations of $L.$
\begin{definition} {\rm\cite{pa}} A representation of a Leibniz algebra $L$ on a vector space $M$ is a pair $(S,T)$ of linear maps
 $S:a\rightarrow S_{a},$ $T:a\rightarrow T_{a}$ of $L$ into $\mathrm{End}(M)$ such that
$$S_{a}\circ S_{b}=S_{[a,b]}-T_{a}\circ S_{b},$$
$$S_{b}\circ T_{a}=T_{a}\circ S_{b}-S_{[a,b]},$$
$$T_{[a,b]}=T_{a}\circ T_{b}-T_{b}\circ T_{a}$$
for all $a,b\in L.$
\end{definition}

The reference \cite{pa} also pointed out that the vector space $M$ equipped with the compositions $ma=S_{a}(m)$ and $am=T_{a}(m)$ is a bimodule of $L.$
 Clearly, the two concepts of representation and bimodule are equivalent.
Let $L$ be a Leibniz algebra. The right multiplication $R_{a}$ (resp., the left multiplication $L_{a}$) of $L$ determined
by any element $a\in L$ is the endomorphism of $L$ defined by $R_{a}(x)=[x,a]$ (resp., $L_{a}(x)=[a,x]$) for all $x\in L.$
The pair $(R,L)$ of linear mappings $R:a\rightarrow R_{a},$ $L:a\rightarrow L_{a}$ is a representation of $L$ on $L$ itself.
In particular, $L:a\rightarrow L_{a}$ is called the adjoint representation of $L.$

\section{The equivalent definition of restricted Leibniz algebras}
\begin{definition}{\rm \cite{sf}}\,  Let $L$ be a Lie algebra over $\mathbb{F}.$ A mapping
$[p]:L\rightarrow L, a\mapsto a^{[p]}$ is called a $p$-mapping, if

$\mathrm{(1)}$  $L_{a^{[p]}}=(L_{a})^{p},\ \forall a\in L.$

$\mathrm{(2)}$  $(\alpha a)^{[p]}=\alpha^{p}a^{[p]},\ \forall a\in L, \alpha\in \mathbb{F}.$

$\mathrm{(3)}$
$(a+b)^{[p]}=a^{[p]}+b^{[p]}+\sum\limits_{i=1}^{p-1}s_{i}(a,b),$\\
where $ (L(a\otimes X+b\otimes 1))^{p-1}(a\otimes 1)=\sum\limits_{i=1}^{p-1}is_{i}(a,b)\otimes X^{i-1} $ in $L\otimes_{\mathbb{F}}\mathbb{F}[X], \forall a,b\in L.$
The pair $(L,[p])$ is referred to as a restricted Lie algebra.
\end{definition}

\begin{definition} \label{d1.1.18}\,   Let $L$ be a Leibniz algebra over $\mathbb{F}.$ A mapping
$[p]:L\rightarrow L, a\mapsto a^{[p]}$ is called a $p$-mapping, if

$\mathrm{(1)}$  $L_{a^{[p]}}=(L_{a})^{p},\ \forall a\in L.$

$\mathrm{(2)}$  $(\alpha a)^{[p]}=\alpha^{p}a^{[p]},\ \forall a\in L, \alpha\in \mathbb{F}.$

$\mathrm{(3)}$
$(a+b)^{[p]}=a^{[p]}+b^{[p]}+\sum\limits_{i=1}^{p-1}s_{i}(a,b),$\\
where $ (L(a\otimes X+b\otimes 1))^{p-1}(a\otimes
1)=\sum\limits_{i=1}^{p-1}is_{i}(a,b)\otimes X^{i-1} $ in
$L\otimes_{\mathbb{F}}\mathbb{F}[X], \forall a,b\in L.$ The pair
$(L,[p])$ is referred to as a restricted Leibniz algebra.
\end{definition}

Clearly, any restricted Lie algebra is a restricted Leibniz  algebra.
Let $L$ be a Leibniz algebra over $\mathbb{F}$ and $f:L\rightarrow L$ be a mapping. $f$ is called a $p$-semilinear mapping, if
 $f(\alpha x+y)=\alpha^{p}f(x)+f(y), \ \forall x,y \in L, \ \forall \alpha\in \mathbb{F}.$ Let $S$ be a subset of
Leibniz algebra $L.$ We put $Z_{L}(S)=\{x\in L|[x,S]=0\}$ and $C_{L}(S)=\{x\in L|[S,x]=0\}.$
$Z_{L}(S)$ and  $C_{L}(S)$ are called the right centralizer of $S$ in $L$ and the left centralizer of $S$ in $L$, respectively.
$Z(L)=\{x\in L|[x,L]=0\}$ is called the right center of $L;$ $C(L)=\{x\in L|[L,x]=0\}$ is called the left center of $L.$ Let $V$ be a subspace of
$L.$ We put $\mathrm{Nor}_{L}(V)=\{x\in L|[V,x]\subseteq V\}.$ $\mathrm{Nor}_{L}(V)$ is called the left normalizer of $V$ in $L.$
 \begin{proposition} \label{p1.3.0}  Let $L$ be a Leibniz algebra over $\mathbb{F}.$ Then the following states hold:

$\mathrm{(1)}$ $I=\langle [x^{[p]^{i}},x^{[p]^{j}}]|x\in L, i,j\in\mathbb{N}]\rangle$
is contained in $Z(L).$

$\mathrm{(2)}$ If $Z(L)=0,$ then $L$ is a Lie algebra.
\end{proposition}
\begin{proof} (1)
For $i,j\in \mathbb{N},$ then
\begin{eqnarray*}&&[[x^{[p]^{i}}, x^{[p]^{j}}], y]\\
&&=[x^{[p]^{i}}, [x^{[p]^{j}}, y]]-[x^{[p]^{j}},[x^{[p]^{i}}, y]]\\
&&=\underbrace{[x,\cdots[x}_{p^{i}}\underbrace{[x\cdots[x}_{p^{j}},y]\cdots ]]\cdots]-\underbrace{[x, \cdots[x}_{p^{j}} \underbrace{[x\cdots [x}_{p^{i}},y]\cdots]]\cdots]\\
&&=0.
\end{eqnarray*}
We have $[x^{[p]^{i}},x^{[p]^{j}}]\in Z(L).$ Consequently, $I=\langle [x^{[p]^{i}},x^{[p]^{j}}]|x\in L, i,j\in\mathbb{N}]\rangle \subseteq Z(L).$

(2) By (1), $[x,x]\in Z(L).$  If $Z(L)=0,$ then  $[x,x]=0.$ Hence $L$ is a Lie algebra.
\end{proof}
\begin{definition} Let $(L,[p])$ be a restricted Leibniz algebra over $\mathbb{F}.$ A subalgebra (ideal or left ideal) $H$
 of $L$ is called a $p$-subalgebra ($p$-ideal or $p$-left ideal) of $L,$
if $x^{[p]}\in H, \forall x\in H.$
\end{definition}
\begin{proposition} \label{p1.3.1} Let $L$ be a subalgebra of a restricted Leibniz
algebra $(G,[p])$ and $[p]_{1}: L\rightarrow L$ a mapping. Then the following
statements are equivalent:

$\mathrm{(1)}$ $[p]_{1}$ is a $p$-mapping on $L.$

$\mathrm{(2)}$ There exists a $p$-semilinear mapping $f:L\rightarrow Z_{G}(L)$
such that $[p]_{1}=[p]+f.$
\end{proposition}
\begin{proof}
(1)$\Rightarrow $(2). Consider $f:L\rightarrow G,$ $f(x)=x^{[p]_{1}}-x^{[p]}.$
Since $L_{f(x)}(y)=0, \forall x,y\in L,$ $f$ actually maps $L$ into $Z_{G}(L).$
For $x,y\in L, \alpha\in \mathbb{F},$ we obtain
\begin{eqnarray*}&&f(\alpha x+y)\\
&&=\alpha^{p}x^{[p]_{1}}+y^{[p]_{1}}+\sum_{i=1}^{p-1}s_{i}(\alpha x,y)-\alpha^{p}x^{[p]}-y^{[p]}-\sum_{i=1}^{p-1}s_{i}(\alpha x,y)\\
&&=\alpha^{p}f(x)+f(y),
\end{eqnarray*}
which proves that $f$ is $p$-semilinear.

(2)$\Rightarrow $(1). We only check the property pertaining to the sum of
two elements  $x,y\in L,$
\begin{eqnarray*}&&
(x+y)^{[p]_{1}}\\
&&=(x+y)^{[p]}+f(x+y)\\
&&=x^{[p]}+f(x)+y^{[p]}+f(y)+\sum_{i=1}^{p-1}s_{i}(x,y)\\
&&=x^{[p]_{1}}+y^{[p]_{1}}+\sum_{i=1}^{p-1}s_{i}(x,y).
\end{eqnarray*}The proof is
complete.
\end{proof}

\begin{corollary} \label{c1.3.3} The following statements hold.

$\mathrm{(1)}$ If $Z(L)=0,$ then $L$ admits at most one $p$-mapping.

$\mathrm{(2)}$ If two $p$-mappings coincide on a basis, then they are equal.

$\mathrm{(3)}$ If $(L,[p])$ is restricted, there exists a $p$-mapping $[p]^{'}$ of $L$ such that
$x^{[p]^{'}}=0, \ \forall x\in Z(L).$
\end{corollary}

\begin{proof}
(1) We set $G=L.$ Then $Z_{G}(L)=Z(L),$ the only $p$-semilinear mapping
occurring in Proposition \ref{p1.3.1} is the zero mapping.

(2) If two $p$-mappings coincide on a basis, their difference vanishes since it is $p$-semilinear.

(3) $[p]|_{Z(L)}$ defines a $p$-mapping on $Z(L).$ Since $Z(L)$ is abelian, it is $p$-semilinear.
Extend this to a $p$-semilinear mapping $f: L\rightarrow Z(L).$ Then $[p]^{'}:=[p]-f$ is a $p$-mapping
of $L,$ vanishing on $Z(L).$
\end{proof}

In the special case of $G=U(L)^{-}\supset L,$ where $U(L)$ is the universal enveloping algebra of $L$ (see \cite{lp}), we obtain

\begin{theorem} \label{t1.3.4} Let $(e_{j})_{j\in J}$ be a basis of $L$
such that there are $y_{j}\in L$ with $(L_{e_{j}})^{p}=L_{y_{j}}.$ Then there exists exactly
one $p$-mapping $[p]:L\rightarrow L$ such that $e_{j}^{[p]}=y_{j},\forall j\in J.$
\end{theorem}
\begin{proof}
For $z\in L,$ we have
$0=((L_{e_{j}})^{p}-L_{y_{j}})(z)=[e^{p}_{j}-y_{j},z].$ Then
$e^{p}_{j}-y_{j}\in Z_{U(L)}{(L)}, \forall j\in J.$ We define a
$p$-semilinear mapping $f:L\rightarrow Z_{U(L)}{(L)}$ by means of
$$f(\sum \alpha_{j}e_{j}):=\sum\alpha_{j}^{p} (y_{j}-e^{p}_{j}).$$
Consider $V:=\{x\in L| x^{p}+f(x)\in L\}.$ The equation
$$(\alpha x+y)^{p}+f(\alpha x+y)=\alpha^{p}x^{p}+y^{p}+\sum^{p-1}_{i=1}s_{i}(\alpha x,y)+\alpha^{p}f(x)+f(y)$$
ensures that $V$ is a subspace of $L.$ Since it contains the basis
$(e_{j})_{j\in J},$ we conclude that $x^{p}+f(x)\in L,\ \forall x\in L.$
By virtue of Proposition \ref{p1.3.1},  $[p]:L\rightarrow L, x^{[p]}:=x^{p}+f(x)$ is a
$p$-mapping on $L.$ In addition, we obtain
$e^{[p]}_{j}=e^{p}_{j}+f(e_{j})=y_{j},$ as asserted. The
uniqueness of $[p]$ follows from Corollary \ref{c1.3.3}.
\end{proof}

\begin{definition} \label{d1.3.400} A Leibniz algebra $L$ is called restrictable, if $L_{L}$ is a $p$-subalgebra of
$\mathrm{Der}(L),$ that is, $(L_{x})^{p}\in L_{L}, \forall x\in L,$ where $L_{L}=\{L_{x}|x\in L\},\ \mathrm{Der}(L)=\{D\in \mathrm{gl}(L)|\ D[x,y]=[D(x),y]+[x,D(y)], \forall x,y\in L\}.$
\end{definition}
\begin{theorem} \label{t1.3.400} $L$ is a restrictable Leibniz algebra if and only if there is a $p$-mapping $[p]:L\rightarrow L$
which makes $L$ a  restricted Leibniz algebra.
\end{theorem}
\begin{proof} $(\Leftarrow)$ By the definition of $p$-mapping $[p],$
we have $(L_{x})^{p}=L_{x^{[p]}}\in L_{g}, \forall x\in L.$ Hence $L$ is restrictable.

$(\Rightarrow)$ Let $L$ be restrictable. Then for $x\in L,$  we have $(L_{x})^{p}\in L_{L},$
that is, there exists $y\in L$  such that $(L_{x})^{p}=L_{y}.$
Let $(e_{j})_{j\in J}$ be a basis of $L.$ Then there exist $y_{j}\in L$ such that $(L_{e_{j}})^{p}=L_{y_{j}}(j\in J).$
By Theorem  $\ref{t1.3.4},$  then there exists exactly one $p$-mapping $[p]: L\rightarrow L$ such that
$e_{j}^{[p]}=y_{j}, \forall j\in J,$ which makes $L$ a restricted  Leibniz algebra.
\end{proof}

\begin{definition}{\rm\cite{das}} \label{d1.1.180} A Leibniz algebra $L$ over $\mathbb{F}$ is called restricted, if for any $x\in L,$
there exists some $x^{[p]}\in L$ such that
$(L_{x})^{p}=L_{x^{[p]}}.$
\end{definition}

\begin{theorem} \label{t1.3.4000}
Definition \ref{d1.1.180} is equivalent to Definition \ref{d1.1.18}.
\end{theorem}
\begin{proof} If $[p]$ satisfies  ${L_{x}}^{p}=L_{x^{[p]}}, \forall x\in L.$ By Definition \ref{d1.3.400}, $L$ is restrictable.
 By Theorem  \ref{t1.3.400}, $L$ satisfies Definition \ref{d1.1.18}.
 Conversely, it is clear. Hence Definition \ref{d1.1.180} is equivalent to Definition \ref{d1.1.18}.
\end{proof}
\begin{remark}
Definition \ref{d1.1.180} is by far more tractable than Definition \ref{d1.1.18}, but just for convenient use it, we give the Definition \ref{d1.1.18}.
\end{remark}

\section{Properties of $p$-mappings and restrictable Leibniz algebras}

One advantage in considering restrictable Leibniz algebras instead of restricted ones rests on the following theorem.
\begin{theorem}  \label{t1.3.6} Let $f:L_{1}\rightarrow L_{2}$ be a surjective homomorphism of Leibniz algebra.
If $L_{1}$ is restrictable, so is $L_{2}.$
\end{theorem}
\begin{proof} Since $f$ is a surjective mapping, one gets $L_{2}=f(L_{1}).$  Then $$(L_{f(x)})^{p}(f(y))=[f(x),\cdots[f(x),f(y)]\cdots ]
=f[x,\cdots[x,y]\cdots ]=f((L_{x})^{p}(y))$$ $$=f(L_{x^{[p]}}(y))=f[x^{[p]},y]=[f(x^{[p]}),f(y)]=L_{f(x^{[p]})}(f(y)), \forall x,y\in L_{1}.$$
Since $L_{1}$ is restrictable, we have $(L_{f(x)})^{p}=L_{f(x^{[p]})}\in L_{L_{2}}.$ Hence $L_{2}$ is restrictable.
\end{proof}

\begin{definition} Let $(L,[p])$ be a restricted Leibniz algebra. A derivation $D$ is called
a restricted derivation, if $D(a^{[p]})=(L_{a})^{p-1}(D(a)).$
\end{definition}

\begin{definition} \label{d1.3.6} Let $A$ be a Leibniz algebra and $B$ be a Lie algebra and
 $\varphi:A\rightarrow \mathrm{Der}(B)$ a homomorphism. On the vector space $A\oplus B,$
 define a multiplication by means of
 $$ [(a,b), (a^{'},b^{'})]:=([a,a^{'}],\varphi(a)(b^{'})-\varphi(a^{'})(b)+[b,b^{'}]).$$
 This algebra, which is denoted by $A\oplus_{\varphi} B,$ is called the semidirect product of $A$ and $B.$
\end{definition}
\begin{theorem} \label{t1.3.7000} Notions such as Definition \ref{d1.3.6}, then
 $A\oplus_{\varphi} B$ is a Leibniz algebra.
\end{theorem}
\begin{proof} Let $(a,b), (a^{'},b^{'}), (a^{''},b^{''})\in A\oplus_{\varphi} B,\; k, k^{'}\in \mathbb{F}.$ Then
\begin{eqnarray*}&&[k(a,b)+k^{'}(a^{'},b^{'}), (a^{''},b^{''})]\\
&&=[(ka+k^{'}a^{'},kb+k^{'}b^{'}), (a^{''},b^{''})]\\
&&=(k[a,a^{''}]+k^{'}[a^{'},a^{''}], \varphi(ka+k^{'}a^{'})(b^{''})-\varphi(a^{''})(kb+k^{'}b^{'})+k[b,b^{''}]+k^{'}[b^{'},b^{''}]).
\end{eqnarray*}On the other hand, one gets
\begin{eqnarray*}&&k[(a,b),(a^{''},b^{''})]+k^{'}[(a^{'},b^{'}),(a^{''},b^{''})]\\
&&=k([a,a^{''}], \varphi(a)(b^{''})-\varphi(a^{''})(b)+[b,b^{''}])+k^{'}([a^{'},a^{''}], \varphi(a^{'})(b^{''})-\varphi(a^{''})(b^{'})\\
&&\quad \quad+[b^{'},b^{''}])\\
&&=(k[a,a^{''}]+k^{'}[a^{'},a^{''}], k\varphi(a)(b^{''})-k\varphi(a^{''})(b)+k[b,b^{''}]+k^{'}\varphi(a^{'})(b^{''})-k^{'}\varphi(a^{''})(b^{'})\\
&&\quad \quad +k^{'}[b^{'},b^{''}])\\
&&=(k[a,a^{''}]+k^{'}[a^{'},a^{''}], \varphi(ka+k^{'}a^{'})(b^{''})-\varphi(a^{''})(kb+k^{'}b^{'})+k[b,b^{''}]+k^{'}[b^{'},b^{''}]).
\end{eqnarray*}Hence
$[k(a,b)+k^{'}(a^{'},b^{'}), (a^{''},b^{''})]=k[(a,b),(a^{''},b^{''})]+k^{'}[(a^{'},b^{'}),(a^{''},b^{''})].$

Note that $\varphi[a,a^{'}]=\varphi(a)\varphi(a^{'})-\varphi(a^{'})\varphi(a).$ Moreover, we have
\begin{eqnarray*}&&[[(a,b),(a^{'},b^{'})],(a^{''},b^{''})]-[(a,b),[(a^{'},b^{'}),(a^{''},b^{''})]]+[(a^{'},b^{'}),[(a,b),(a^{''},b^{''})]]\\
&&=[([a,a^{'}], \varphi(a)(b^{'})-\varphi(a^{'})(b)+[b,b^{'}]),(a^{''},b^{''})]-[(a,b), ([a^{'},a^{''}], \varphi(a^{'})(b^{''})\\&&  \quad-\varphi(a^{''})(b^{'})+[b^{'},b^{''}])]+[(a^{'},b^{'}), ([a,a^{''}], \varphi(a)(b^{''})-\varphi(a^{''})(b)+[b,b^{''}])] \\
&&=([[a,a^{'}],a^{''}],\varphi[a,a^{'}](b^{''})-\varphi(a^{''})(\varphi(a)(b^{'})-\varphi(a^{'})(b)+[b,b^{'}])+[\varphi(a)(b^{'}),b^{''}]
\\&&  \quad\quad \quad \quad \quad -[\varphi(a^{'})(b),b^{''}]
 +[[b,b^{'}],b^{''}])\\&&  \quad  -([a,[a^{'},a^{''}]],\varphi(a)(\varphi(a^{'})(b^{''})-\varphi(a^{''})(b^{'})+[b^{'},b^{''}])-\varphi[a^{'},a^{''}](b)+[b,\varphi(a^{'})(b^{''})]
\\&&  \quad\quad \quad \quad \quad  -[b,\varphi(a^{''})(b^{'})]+[b,[b^{'},b^{''}]])\\&&  \quad +([a^{'},[a,a^{''}]],\varphi(a^{'})(\varphi(a)(b^{''})-\varphi(a^{''})(b)+[b,b^{''}])-\varphi[a,a^{''}](b^{'})+[b^{'},\varphi(a)(b^{''})]\\&&  \quad \quad \quad \quad \quad
-[b^{'},\varphi(a^{''})(b)]+[b^{'},[b,b^{''}]])\\
&&=(0, \varphi[a,a^{'}](b^{''})-\varphi(a^{''})\varphi(a)(b^{'})+\varphi(a^{''})\varphi(a^{'})(b)-\varphi(a^{''})[b,b^{'}]+[\varphi(a)(b^{'}),b^{''}]\\&&   \quad-[\varphi(a^{'})(b),b^{''}]
+[[b,b^{'}],b^{''}]-\varphi(a)\varphi(a^{'})(b^{''})+\varphi(a)\varphi(a^{''})(b^{'})-\varphi(a)[b^{'},b^{''}]\\&&  \quad+\varphi[a^{'},a^{''}](b)-[b,\varphi(a^{'})(b^{''})]
+[b,\varphi(a^{''})(b^{'})]-[b,[b^{'},b^{''}]] +\varphi(a^{'})\varphi(a)(b^{''})\\&&  \quad-\varphi(a^{'})\varphi(a^{''})(b)+\varphi(a^{'})[b,b^{''}]-\varphi[a,a^{''}](b^{'})+[b^{'},\varphi(a)(b^{''})]
-[b^{'},\varphi(a^{''})(b)]\\&&  \quad+[b^{'},[b,b^{''}]])
\\&&=0.
\end{eqnarray*}
As a result, $A\oplus_{\varphi} B$ is a Leibniz algebra. The result follows.
\end{proof}
\begin{theorem} \label{t1.3.7} Let $(A,[p])$ be a restricted Leibniz algebra and $(B,[p])$ be a restricted Lie algebra. If  $\varphi:A\rightarrow \mathrm{Der}(B)$  be restricted  homomorphism such that $\varphi(x)$ is restricted for every $x\in A,$ then $A\oplus_{\varphi} B$ is restrictable.
\end{theorem}
\begin{proof} Let $x\in A.$ Then $(L_{x})^{p}-L_{x^{[p]}}|_{A}=0$ and $(L_{x})^{p}-L_{x^{[p]}}|_{B}=\varphi(x)^{p}-\varphi(x^{[p]})=0$ holds,
hence $(L_{x})^{p}\in L_{A\oplus_{\varphi} B},\ \forall x\in A.$ If $x\in B,$ then $(L_{x})^{p}-L_{x^{[p]^{'}}}|_{B}=0$
and for $y\in A,$ we obtain
$$((L_{x})^{p}-L_{x^{[p]^{'}}})(y)=-(L_{x})^{p-1}\circ \varphi(y)(x)+\varphi(y)(x^{[p]^{'}})=0,$$
hence $(L_{x})^{p}\in L_{A\oplus_{\varphi} B},\ \forall x\in B.$ Therefore, $A\oplus_{\varphi} B$ is restrictable by Theorem \ref{t1.3.4}.
\end{proof}

\begin{corollary} \label{c1.3.8} Let $A,B$ be ideals of a Leibniz algebra $L$ such that $L=A\oplus B.$
Then $L$ is restrictable if and only if $A,B$ are restrictable.
\end{corollary}
\begin{proof}
 If $A,B$ are  restrictable, by Theorem \ref{t1.3.7} and setting $\varphi=0,$  we conclude that $L$ is restrictable.
 If $L$ is restrictable , so are $A\cong L/B,$ $B\cong L/A$ by Theorem \ref{t1.3.6}.
 \end{proof}

 \begin{corollary} \label{c1.3.9} Let $A,B$ be restrictable ideals of a Leibniz algebra $L$ such that $L=A+B$
 and $[A,B]=[B,A]=0.$ Then $L$ is restrictable.

\end{corollary}
\begin{proof}
 Define a mapping $f:A\oplus B\rightarrow L, (x,y)\mapsto x+y.$
Clearly, $f$ is a surjective  homomorphism.
For $(x_{1},y_{1}), (x_{2}, y_{2})\in A\oplus B,$
by $[A,B]=[B,A]=0,$ one gets $[x_{1},y_{2}]=[y_{1},x_{2}]=0.$ We have
\begin{eqnarray*}&&
f[(x_{1},y_{1}), (x_{2}, y_{2})]=f([x_{1},y_{1}], [x_{2}, y_{2}])\\
&&=[x_{1},x_{2}]+[y_{1},y_{2}]=[x_{1},x_{2}]+[x_{1},y_{2}]+[y_{1},x_{2}]+[y_{1},y_{2}]\\
&&=[x_{1}+y_{1},x_{2}+y_{2}]=[f(x_{1},y_{1}),f(x_{2}, y_{2})].
\end{eqnarray*}
By Corollary \ref{c1.3.8}, we have $A\oplus B$ is restrictable.
By Theorem \ref{t1.3.6}, one gets $L$ is restrictable.
\end{proof}
\begin{definition} Let $L$ be a Leibniz algebra and $\psi$ be a symmetric bilinear form on $L.$
$\psi$ is called associative, if $\psi([x,z],y)=\psi(x,[z,y]).$
\end{definition}

\begin{definition} Let $L$ be a Leibniz algebra and $\psi$ a symmetric bilinear form on $L.$
Set $L^{\bot}=\{x\in L| \psi(x,y)=0, \forall \ y\in L\}.$ $L$ is called nondegenerate, if $L^{\bot}=0.$
\end{definition}

\begin{theorem} \label{l1.3.16} Let $L$ be a subalgebra of the restricted Leibniz algebra $(G,[p]).$  Assume
 $\lambda:G\times G\rightarrow \mathbb{F}$ to be an associative symmetric bilinear form, which is nondegenerate on $L\times L.$
 Then $L$ is restrictable.
\end{theorem}

\begin{proof} Since $\lambda$ is nondegenerate on $L\times L,$ every linear form $f$ on $L$ is determined
by a suitably chosen element $y\in L:f(z)=\lambda(y,z), \forall z\in L.$ Let $x\in L.$
Then there exists $y\in L$ such that
$$\lambda(x^{[p]},z)=\lambda(y,z),\forall z\in L.$$
This implies that $0=\lambda (x^{[p]}-y,L^{(1)})=\lambda ([x^{[p]}-y,L],L)$ and $[x^{[p]}-y,L]=0.$
Therefore, we have
$$(L_{x}|_{L})^{p}=L_{x^{[p]}}|_{L}=L_{y}|_{L},$$
proving that $L$ is  restrictable.
\end{proof}
\begin{corollary} \label{c1.3.13} Let $(S,T)$ be a finite-dimensional  representation of $L$
such that $k_{T}$ is nondegenerate on $L\times L,$ where $S:L\rightarrow \mathrm{End}(M)$ and $T:L\rightarrow \mathrm{End}(M).$ Then $L$ is restrictable.
\end{corollary}
\begin{proof}
 The associative form $(x,y)\mapsto \mathrm{tr}(x,y)$ on $\mathrm{End}(M)\times \mathrm{End}(M)$ is nondegenerate on $T(L)\times T(L).$
 Hence $T(L)$ is restrictable, since $T$ is faithful, $L$ is restrictable.
\end{proof}

\begin{proposition} \label{p1.3.14} Let $L$ be a restrictable Leibniz algebra and $H$ a subalgebra of $L.$ Then $H$ is a
$p$-subalgebra for some mapping $[p]$ on $L$ if and only if $(L_{H}|_{L})^{p}\subseteq L_{H}|_{L}.$
\end{proposition}
\begin{proof}
$(\Rightarrow)$ If $H$ is a
$p$-subalgebra, then $x^{[p]}\in H,\ \forall x\in H.$
$(L_{x})^{p}=L_{x^{[p]}}\subseteq L_{H}|_{L}.$
Hence, $(L_{H}|_{L})^{p}\subseteq L_{H}|_{L}.$

$(\Leftarrow)$ If $(L_{H}|_{L})^{p}\subseteq L_{H}|_{L},$ then $H$ is restrictable. By Theorem  \ref{t1.3.400},
$H$ is restricted. Thereby, $H$ is a $p$-subalgebra of $L.$
\end{proof}

\begin{proposition} \label{p1.3.15} Let $L,$ $L^{'}$ be restrictable Leibniz algebras and $f: L\rightarrow L^{'}$ a surjective
homomorphism. If $Z(L^{'})=0,$ then $\mathrm{ker}(f)$ is a $p$-ideal for every $p$-mapping on $L.$
\end{proposition}

\begin{proof} Clearly, $\mathrm{ker}(f)\lhd L.$
Since $L$ is restrictable, there exists $y\in L$ such that $(L_{x})^{p}=L_{y}, \forall x\in \mathrm{ker}(f).$
For $z\in L,$ we have $(L_{x})^{p}(z)=L_{y}(z).$
i.e., $[x,\cdots [x,[x,z]]\cdots ]=[y,z].$ Since $f$ is a homomorphism mapping, $[f(x),\cdots[f(x),[f(x),f(z)]]\cdots ]=[f(y),f(z)],$
that is, $(L_{f(x)})^{p}(f(z))=L_{f(y)}(f(z)).$
Since $f$ is a surjective mapping, one gets $L^{'}=\{f(z)| z\in L\},$ hence
$(L_{f(x)})^{p}=L_{f(y)}.$ By Theorem \ref{t1.3.6}, we have $L^{'}$ is restrictable.
Moreover, $L_{f(x)^{[p]^{'}}}=L_{f(y)}.$ By $Z(L^{'})=0,$ one gets $f(y)=f(x)^{[p]^{'}}=0, y\in \mathrm{ker}(f).$
 $(L_{x})^{p}=L_{y}\in L_{\mathrm{ker}(f)},$ that is, $\mathrm{ker}(f)$ is restrictable.
Therefore, $\mathrm{ker}(f)\lhd_{p}L.$
\end{proof}
\begin{theorem} \label{p1.3.16} Let $(L,[p])$ be a restricted Leibniz algebra and $D$ a derivation. Then
$D(x^{[p]})-(L_{x})^{p-1}(D(x))\in Z(L), \forall x\in L.$
\end{theorem}

\begin{proof}
Let $D\in \mathrm{Der}(L)$ and $a,x\in L.$ If $A$ is the transformation $x\mapsto [a,x]$ and $B$ is the transformation $x\mapsto [D(a),x],$ then
 $A=L_{a},$  $B=L_{D(a)}.$  We can prove $(L_{A})^{k}(B)=\sum_{i=0}^{k}(-1)^{k-i}C_{k}^{i}A^{i}BA^{k-i}$
by induction on $k.$

Then by the result, we have
$$(L_{A})^{p-1}(B)=\sum\limits_{i=0}^{p-1}(-1)^{p-1-i}C_{p-1}^{i}A^{i}BA^{p-1-i}.$$
Since $$C_{p-1}^{i}=\frac{(p-1)(p-2)\cdots (p-i)}{i\cdot(i-1)\cdots
1}=\frac{(-1)(-2)\cdots (-i)}{i\cdot(i-1)\cdots 1}=(-1)^{i},$$ we
have $(-1)^{p-1-i}C_{p-1}^{i}=(-1)^{p-1}=1.$ So
$$BA^{p-1}+ABA^{p-2}+\cdots+A^{p-1}B=[A,\cdots[A,B]\cdots].$$ Then
\begin{eqnarray*}&&
D[a^{[p]},x]=D[a,\cdots[a,x]\cdots]\\
&&=[D(a),\cdots[a,x]\cdots]+\cdots +[a,\cdots[a,D(x)]\cdots]\\
&&=[a^{[p]},D(x)]+[a,a\cdots[a,D(a)]\cdots x].
\end{eqnarray*}
On the other hand, we have $D[a^{[p]},x]=[D(a^{[p]}),x]+[a^{[p]},D(x)]$ since $D$ is a derivation. Hence $[D(a^{p}),x]=[a,a\cdots [a,D(a)]\cdots x]$ for all $x\in L,$
that is, $D(a^{[p]})-(L_{a})^{p-1}(D(a))\in Z(L), \forall a\in L.$
\end{proof}
\begin{corollary} \label{c1.3.16} Let $(L,[p])$ be a restricted Leibniz algebra. If $Z(L)=0,$
then every derivation of $L$ is a restricted  derivation.
\end{corollary}

\begin{corollary} \label{c1.3.160} Let $(L,[p])$ be a restricted Lie algebra with trivial center.
Then every derivation of $L$ is a restricted  derivation.
\end{corollary}
Let $S\subseteq L$ be a subset of a restricted Leibniz algebra
$(L,[p]).$ The intersection of all $p$-subalgebras containing $S$
will be denoted by $S_{p}.$ $S_{p}$ is a $p$-subalgebra generated
by $S$ in $L.$ By definition, $S_{p}$ is the smallest
$p$-subalgebra of $(L,[p])$ containing $S.$

We propose to give a more explicit characterization of $S_{p}$ in
some special cases. The image of $S$ under the iterated
application of the $p$-mapping $[p]$ will be denoted by
$S^{[p]^{i}},$ that is, $S^{[p]^{i}}:=\{x^{[p]^{i}}|x\in S \}.$

\begin{proposition}\label{p3.1.1} Let $(L,[p])$ be a restricted Leibniz algebra over $\mathbb{F}$ and $H\subseteq L$
a left ideal. Suppose that $(e_{j})_{j\in J}$ is a basis of $H.$
Then

$\mathrm{(1)}$ $ H_{p}=\sum_{i\in \mathbb{N}}\langle
H^{[p]^{i}}\rangle=\sum_{j\in J,i\in
\mathbb{N}}\mathbb{F}e_{j}^{[p]^{i}}.$

$\mathrm{(2)}$ $[H_{p},L]=[H,L];$ $(H_{p})^{n}=H^{n},$
$(H_{p})^{(n)}=H^{(n)}, n\geq 1.$

$\mathrm{(3)}$ $H_{p}$ is solvable (nilpotent) if and only if $H$ is solvable
(nilpotent).

$\mathrm{(4)}$ $H_{p}$ is a $p$-left ideal.
\end{proposition}

\begin{proof} (1) Put $G:=\sum_{j\in J,i\in
\mathbb{N}}\mathbb{F}e_{j}^{[p]^{i}}.$ Then, clearly,
$H\subseteq G\subseteq \sum_{i\geq 0}\langle H^{[p]^{i}}\rangle
\subseteq H_{p}.$ To prove $H_{p}\subseteq G,$ we observe that by property
(1) of the definition of $p$-mapping, $[e_{k}^{[p]^{i}},$
$e_{l}^{[p]^{j}}]$ $=L_{e_{k}^{[p]^{i}}}(e_{l}^{[p]^{j}})$
$=(L_{e_{k}})^{p^{i}}(e_{l}^{[p]^{j}})$
$=(L_{e_{k}})^{p^{i}-1}[e_{k},e_{l}^{[p]^{j}}]\in H^{2}\subseteq
H\subseteq G.$ Hence $G$ is a subalgebra. Put $V=\{x\in G |
x^{[p]}\in G \}.$ Since $G$ is a subalgebra, (2) and
(3) of the definition of $p$-mapping prove that $V$ is a subspace
containing the generating set $\{e_{j}^{[p]^{i}}|j\in
J,i\geq 0\}$ of $G.$ Hence $V=G$ and $G$ is closed under
the $p$-mapping. Consequently, $G$ is a $p$-subalgebra containing $H$
and $H_{p}\subseteq G.$

(2) Considering (1) and (2) of the definition, we get $[H_{p},L]\subseteq [H,L].$

(3) It follows from (2).

(4) By (2), we have $[H_p,
L]=[H,L]\subseteq H \subseteq H_{p},$ $H_{p}$ is a left-ideal of $L.$  Moreover, $H_{p}$ is a $p$-subalgebra of $L.$ Hence
$H_{p}$ is a $p$-left ideal of $L.$
\end{proof}

\section{Restricted Leibniz algebras whose elements are \\ semisimple}

\begin{definition} Let $(L,[p])$ be a restricted Leibniz algebra over $\mathbb{F}.$ An element $x\in L$ is called semisimple
if $x=\sum\limits_{i=1}^{m}\alpha_{i}x^{[p]^{i}}$ and toral if
$x^{[p]}=x.$
\end{definition}

\begin{proposition} \label{p2.1.1} Let $(L,[p])$ be a restricted Leibniz algebra over $\mathbb{F}.$
Then the following statements hold:

$\mathrm{(1)}$ Every toral element is semisimple.

$\mathrm{(2)}$ If $x$ is semisimple, then $T(x)$ is semisimple for every finite-dimensional $p$-represen-\\tation $(S,T),$ where $S:L\rightarrow \mathrm{End}(M), T:L\rightarrow \mathrm{End}(M).$

$\mathrm{(3)}$ If $\mathbb{F}$ is perfect and $[p]$ is nonsingular, then every element $x\in L$ is semisimple.

$\mathrm{(4)}$ An endomorphism $\sigma \in \mathrm{End}(M)$ is semisimple if and only if it is semisimple as an element
of the restricted Leibniz algebra $(\mathrm{gl}(M),p).$
\end{proposition}

\begin{proof}
(1) Clearly.

(2) Let $(S,T)$ be a finite-dimensional $p$-representation. Then $T(x^{[p]})=T(x)^{p}$
and the semisimplicity of $x$ ensures the existence of
 $\alpha_{1},\cdots,\alpha_{n}\in \mathbb{F}$ such that $T(x)=\sum\limits_{i=1}^{n}\alpha_{i}T(x)^{p^{i}}.$
Let $m_{x}$ be the minimum polynomial of $T(x).$ Then there is
$\lambda\in \mathbb{F}[x]$ such that $\lambda
m_{x}=\sum\limits_{i=1}^{n}\alpha_{i}x^{p^{i}}-x.$ Taking the
derivative we obtain $\lambda^{'}m_{x}+\lambda m_{x}^{'}=-1,$ which
means that $T(x)$ is semisimple.

(3) Let $x$ be an element of $L\backslash \{0\}.$ $L$ is  finite-dimensional and there is a minimal element $m\in
\mathbb{N}\backslash \{0\}$ such that $x^{[p]^{m}}\in
\langle\{x,\cdots,x^{[p]^{m-1}}\}\rangle.$ The set
$\{x,\cdots,x^{[p]^{m-1}}\}$ is therefore linearly independent. We
find $\alpha_{1}, \cdots, \alpha_{m}\in \mathbb{F}$ such that
$x^{[p]^{m}}=\sum\limits_{i=1}^{m}\alpha_{i}x^{[p]^{i-1}}.$ The
assumption $\alpha_{1}=0$ forces
$x^{[p]^{m-1}}-\sum\limits_{i=2}^{m}\alpha_{i}^{1/p}x^{[p]^{i-2}}$
to be a zero of $[p],$ thus $x^{[p]^{m-1}}\in
\langle\{x,\cdots,x^{[p]^{m-2}}\}\rangle.$ This contradicts the
choice of $m.$ We have $\alpha_{1}\neq 0.$ Thus
$x=\alpha_{1}^{-1}x^{[p]^{m}}-\sum\limits_{i=2}^{m}\alpha_{i}\alpha_{1}^{-1}x^{[p]^{i-1}}.$
Hence $x$  is semisimple.

(4) If $\sigma$ is a semisimple element of $(\mathrm{gl}(M),p),$ then (2) entails the semisimplicity of $\sigma.$ Assume conversely that $\sigma$ is semisimple.
Let $\mathbb{\overline{F}}$ denote an algebraic closure of $\mathbb{F}.$ Then $\bar{\sigma}:=\sigma\otimes \mathrm{id}_{\mathbb{\overline{F}}}$
is a diagonalizable endomorphism of $M\otimes_{\mathbb{F}}\mathbb{\overline{F}}.$ Consequently, $\mathbb{\overline{F}}[\overline{\sigma}]\subseteq \mathrm{End}_{\mathbb{F}}(M\otimes_{\mathbb{F}}\mathbb{\overline{F}})$
does not contain any nonzero nilpotent elements. On the basis of Lemma \ref{l2.1.1} (4), this implies, as $\overline{\mathbb{F}}$ is perfect, the surjectivity of $p: \overline{\mathbb{F}}[\overline{\sigma}]\rightarrow \overline{\mathbb{F}}[\overline{\sigma}].$ Hence $\langle \mathbb{F}[\sigma^{p}]\rangle=\mathbb{F}[\sigma],$ as desired.
\end{proof}

\begin{theorem} \label{t2.1.70} Let $(L,[p])$ be a restricted Leibniz algebra over $\mathbb{F}.$ For every $x\in L,$ there exists $k\in \mathbb{N}\backslash \{0\}$ such that $x^{[p]^{k}}$ is semisimple.
\end{theorem}
\begin{proof} The family $(x^{[p]^{i}})_{i\geq 0}$ is linearly dependent. Then there exist
$k\geq 0,$ $\alpha_{1}, \cdots, \alpha_{n}\\\in \mathbb{F}$ such
that $x^{[p]^{k}}=\sum\limits_{i=1}^{n}\alpha_{i}x^{[p]^{k+i}}.$
This means that $x^{[p]^{k}}$ is semisimple.
\end{proof}
\begin{proposition} \label{t2.1.7}
 Let $(L,[p])$ be a restricted Leibniz algebra over a perfect field $\mathbb{F}.$ Then the following states hold:

$\mathrm{(1)}$ $[p]$ is injective if and only if $[p]$ is nonsingular.

$\mathrm{(2)}$ If $[p]$ is nonsingular, then $[p]$ is surjective.
\end{proposition}
\begin{proof}
(1) Let $[p]$ be injective and $0\neq x\in L.$ If $x^{[p]}=0.$ Then $x\in \mathrm{ker}[p],$ which implies $\mathrm{ker}[p]\neq 0.$ This is a contradiction. Hence
$x^{[p]}\neq 0.$ i.e., $[p]$ is nonsingular. Conversely, $x^{[p]}=0, \forall x\in \mathrm{ker}[p].$ Since $[p]$ is nonsingular, then $x=0.$ Hence $\mathrm{ker}[p]=0,$ i.e., $[p]$ is injective.

(2) Suppose that $[p]$ is nonsingular. Let $x$ be an element of $L.$
Using Proposition \ref{p2.1.1} (3), we conclude that $x$ is
semisimple,
$x=\sum\limits_{i=1}^{m}\alpha_{i}x^{[p]^{i}}=(\sum\limits_{i=1}^{m}\alpha^{1/p}_{i}x^{[p]^{i-1}})^{[p]}.$
 We  get $ \sum\limits_{i=1}^{m}\alpha^{1/p}_{i}x^{[p]^{i-1}}\in L,$ since $\mathbb{F}$ is perfect.
Hence $x$ is an image under $[p].$
\end{proof}

\begin{theorem} \label{t2.1.700}
 Let $(L,[p])$ be a restricted Leibniz algebra over a perfect field $\mathbb{F}.$ Then the following states are equivalent:

$\mathrm{(1)}$ Every element of $L$ is semisimple.

$\mathrm{(2)}$ $[p]$ has no nontrivial zero.

$\mathrm{(3)}$ $[p]$ is nonsingular.
\end{theorem}
\begin{proof}
(1)$\Rightarrow$ (3) Let $x\in L\backslash \{0\}.$ By (1),
$x=\sum\limits_{i=1}^{m}\alpha_{i}x^{[p]^{i}}.$ If $x^{[p]}=0,$ then
$x=0,$ this is a contradiction, hence $x^{[p]}\neq 0,$  i.e., $[p]$
is nonsingular.

(3)$\Rightarrow$ (1) By Proposition $\ref{p2.1.1}$ (3), (1) holds.

(2)$\Leftrightarrow$ (3) Clearly.
\end{proof}

\begin{proposition} \label{p2.1.8} Let $\mathbb{F}$ be  perfect and $(L,[p])$ a restricted Leibniz algebra
such that $[p]$ is nonsingular. Then for any  $p$-subalgebra $H$ of $L$ the implication
$$x^{[p]^{r}}\in H\Rightarrow x\in H$$
holds.
\end{proposition}

\begin{proof}
 Since $[p]$ is nonsingular on $H,$ by Proposition \ref{t2.1.7}, $[p]$ is surjective. We therefore find $y\in H$ such that $x^{[p]^{r}}=y^{[p]^{r}}.$
Since $[p]$ is injective on $L,$ we conclude that $x=y\in H.$
\end{proof}

\begin{proposition} \label{t4.1.1}Let $\mathbb{F}$ be  perfect and $(L,[p])$ a solvable restricted Leibniz algebra
with a nonsingular $[p]$-mapping. Then $L$ is abelian.
\end{proposition}

\begin{proof}
Let $n\geq 0$ be the minimal integer with respect to the condition $L^{(n)}\subseteq Z(L).$ If $n> 0,$ then $x^{[p]^{2}}\in Z(L)$ holds for any $x\in L^{(n-1)}.$
Hence by Proposition \ref{p2.1.8}, $x\in Z(L),$ contradicting the choice of $n.$ Therefore, we obtain $n=0$ and so $L=L^{(0)}=Z(L).$
\end{proof}

\section{Tori and Cartan decomposition}
\begin{definition}\rm \cite{das} \label{d2.1.8} Let $Q$ be a Leibniz algebra and $M$ be a $Q$-module. $M$ is called symmetric, if $[x,m]+[m,x]=0,$ for any $x\in Q, m\in M.$
\end{definition}
Accordingly, we have the following definition.
\begin{definition} A representation $(S,T)$ of a Leibniz algebra $L$ on the vector space $M$ is called symmetric,
 if $S_{a}+T_{a}=0$ for all $a\in L.$
\end{definition}
We have the following Theorems \ref{t2.2.4} and \ref{t2.2.400}, whose proofs are analogous to restricted Lie algebra(cf.\cite{sf}).
\begin{theorem} \label{t2.2.4} Let $H$ be a nilpotent Leibniz algebra and $(S,T)$ a finite-dimensional symmetric representation, where $S:H\rightarrow \mathrm{End}(M), T:H\rightarrow \mathrm{End}(M).$ Then there exists a finite set $B\subseteq \mathrm{Map}(H,\mathbb{F}[\chi])$
such that

$\mathrm{(1)}$ $\pi_{h}$ is irreducible, $\forall \pi \in B, \forall h\in H.$

$\mathrm{(2)}$ $M_{\pi}$ is an $H$-submodule, $\forall \pi\in B.$

$\mathrm{(3)}$ $M=\bigoplus_{\pi\in B}M_{\pi}.$
\end{theorem}
\begin{definition} A nilpotent subalgebra $H$ of a Leibniz algebra $L$ is a Cartan subalgebra, if $\mathrm{Nor}_{L}(H)=H.$
\end{definition}
\begin{theorem} \label{t2.2.400} Let $H$ be a Cartan subalgebra of Leibniz algebra $L$
over an algebraically closed field $\mathbb{F}.$  Then $L$ has the decomposition $L=\bigoplus_{\alpha\in\Phi}L_{\alpha},$ which is
    referred to as the root space decomposition of $L$ relative to $H.$
\end{theorem}

\begin{definition} Let $(L,[p])$ be a restricted Leibniz algebra over $\mathbb{F}.$ A subalgebra $T\subseteq L$ is called a torus if

$\mathrm{(1)}$ $T$ is an abelian $p$-subalgebra.

$\mathrm{(2)}$ $x$ is semisimple, $\forall x\in T.$
\end{definition}

\begin{remark} \label{r5.7}
Suppose that $h\in L$ is a semisimple element which acts nilpotently on an element $x\in L.$ That is, there is $n\in\mathbb{N}\backslash \{0\},$ $(L_{h})^{n}(x)=0.$ In fact, the semisimplicity of $h$ readily yields that $h=\sum_{i\geq 0}\alpha_{i}h^{[p]^{k+i}}.$
Choose $k\in \mathbb{N}\backslash \{0\}$ such that $p^{k}\geq n.$ Then $[\sum_{i\geq 0}\alpha_{i}h^{[p]^{k+i}}, x]=0$  and $L_{h}(x)=0.$
\end{remark}

\begin{theorem}  \label{t2.3.3} Let $(L,[p])$ be a restricted Leibniz algebra and $H\subseteq L$ a subalgebra of $L$.
If there exists a maximal torus $T\subseteq L$ such that $H=Z_{L}(T)=C_{L}(T),$ then $H$ is a Cartan subalgebra of $L.$
\end{theorem}
\begin{proof} Assume that $T$ is a maximal torus. Let $x$ be an element of $H=Z_{L}(T)=C_{L}(T).$  There is $ k\in \mathbb{N}\backslash \{0\}$ such that
$x^{[p]^{k}}$  is  semisimple. Since $x^{[p]^{k}}$ is contained in the $p$-subalgebra $H,$ $T_{1}:=T+\mathbb{F}x^{[p]^{k}}$ is a torus of $L$
containing $T.$ In fact, $[T, x^{[p]^{k}}]=[x^{[p]^{k}}, T]=0,$ since $H=Z_{L}(T)=C_{L}(T).$ Hence $T_{1}$ is abelian.
Let $z=y+\beta x^{[p]^{k}}\in T_{1}(y\in T, \beta\in \mathbb{F}).$ Then
$z^{[p]}=y^{[p]}+\beta^{p} x^{[p]^{k+1}}\in T+\mathbb{F}x^{[p]^{k}},$
since $ x^{[p]^{k}}$ is semisimple and $T$ is a $p$-subalgebra. Hence $T_{1}$ is a $p$-subalgebra.
Let $y\in T.$ Consider the $p$-mapping on $V:=(\mathbb{F}y+\mathbb{F}x^{[p]^{k}})_{p}.$
 $[p]:V\rightarrow V$ is $p$-semilinear, since $T_{1}$ is abelian. The semisimplicity of $y$ and $x^{[p]^{k}}$
show that $y, x^{[p]^{k}}\in \langle V^{[p]}\rangle.$ Hence $\langle V^{[p]}\rangle=V.$  Then $y+\beta x^{[p]^{k}}$
is semisimple follows from Lemma \ref{l2.1.2}. Clearly, $T\subseteq T_{1}.$
 The maximality of $T$ then shows that $x^{[p]^{k}}\in T.$ Consequently, $(L_{x})^{p^{k}}(H)=0,$ proving that $L_{x}|_{H}$
is nilpotent. By Engel's theorem(cf.\cite[Theorem 1.1]{b}), $H$ is nilpotent.
Let $x$ be an element of $\mathrm{Nor}_{L}(H).$ Then
$(L_{h})^{2}(x)=0$ for every $h\in T.$ Since $h\in T$ is semisimple, by remark $\ref{r5.7},$
we obtain $L_{h}(x)=0,\ \forall h\in T,$ hence $x\in C_{L}(T)=H.$ As
a result, $H$ is a  Cartan subalgebra of $L.$
\end{proof}

\begin{proposition}\label{c2.3.4}
Let $(L,[p])$ be a restricted Leibniz algebra. If $H$ is nilpotent,
then $T:=\{h\in Z(H)| h$ semisimple $\}$ is a maximal torus of $H.$
\end{proposition}

\begin{proof} Let $x,y$ be two elements of $T.$ Then $x\in Z(H),$ $y\in Z(H).$ We have $[y,x]=[x,y]=0.$
Hence $T$ is abelian. Since $Z(H)$ is a $p$-subalgebra, $x^{[p]}\in Z(H).$  Since $x$ is  semisimple, so is $x^{[p]}.$ Hence $x^{[p]}\in T.$
$T$ is a $p$-subalgebra. Hence, $T$ is  a torus of $H.$

Let $T^{'}$ be a torus of $H$ and $T\subseteq T^{'}.$ Let $x\in T^{'}.$ Then $x\in H.$ Since $H$ is nilpotent, there exists $n\in \mathbb{N}\backslash \{0\}$ such that $(L_{x})^{n}=0,$ $(L_{x})^{n}(h)=0, \forall h\in H.$ Since $x$ is  semisimple,  by remark $\ref{r5.7},$ one gets $L_{x}(h)=0, \ x\in Z(H),\ x\in T.$ Then $T^{'}=T.$
$T$ is a maximal torus.
 \end{proof}

\begin{corollary}\label{c2.3.5}
Let $(L,[p])$ be a restricted Leibniz algebra over an algebraically closed field $\mathbb{F}.$ Consider the root space
decomposition $L=\bigoplus_{\alpha\in \Phi}L_{\alpha}$ with respect to a Cartan subalgebra $H.$ Then the following states hold:

$\mathrm{(1)}$ If $h\in H$ is  semisimple, then $L_{h}|_{L_{\alpha}}=\alpha(h)\mathrm{id}_{L_{\alpha}};$ $\alpha(h)\in \mathrm{GF}(p)$ for all toral $h\in H,$ where $\mathrm{GF}(p)$ is a finite field.

$\mathrm{(2)}$ $\alpha(x^{[p]})=0,  \forall x\in L_{\alpha}.$
\end{corollary}

\begin{proof} (1) $L_{h}|_{L_{\alpha}}$ is  semisimple and consequently diagonalizable. Since $\alpha(h)$ is the only eigenvalue of $L_{h}|_{L_{\alpha}},$
we obtain $L_{h}|_{L_{\alpha}}=\alpha(h)\mathrm{id}_{L_{\alpha}}.$ Suppose that $h$ is toral. Then
$\alpha(h) \mathrm{id}_{L_{\alpha}}=\alpha(h^{[p]})\mathrm{id}_{L_{\alpha}}=
L_{h^{[p]}}|_{L_{\alpha}}=(L_{h})^{p}|_{L_{\alpha}}=\alpha(h)^{p}\mathrm{id}_{L_{\alpha}}.$ This proves that $\alpha(h)=\alpha(h)^{p}$
and $\alpha(h)\in \mathrm{GF}(p).$

(2) Let $x$ be a nonzero element of $L_{\alpha}.$ Then $[x^{[p]},x]=0$ and $0$ is an eigenvalue of $L_{x^{[p]}}|_{L_{\alpha}}.$
As $\alpha(x^{[p]})$ is the only eigenvalue of $L_{x^{[p]}}|_{L_{\alpha}},$ we obtain $\alpha(x^{[p]})=0.$
 \end{proof}

\begin{lemma} \label{l2.3.6}
Let $T$ be a torus of the restricted Leibniz algebra $(L,[p]).$

$\mathrm{(1)}$ Any $T$-invariant subspace $W\subseteq L$(i.e.,$[T,W]\subseteq W$) decomposes $W=C_{W}(T)+[T,W].$

$\mathrm{(2)}$ If $I\lhd_{p}L$ is a $p$-ideal such that $L/I$ is a torus, then there exists a torus $T^{'}\supset T$
such that $L=T^{'}+I.$
\end{lemma}
\begin{proof}
(1) The adjoint representation gives $W$ the structure of a $T$-module. According to Theorem \ref{t2.2.4}, we may write
$W=\oplus_{\pi\in B}W_{\pi}.$  Let $\pi_{0}$ be the function with $\pi_{0h}=X, \ \forall h\in T.$
Then $W_{\pi_{0}}\subseteq C_{W}(T)$ and $[T,W_{\pi}]=W_{\pi}, \forall \pi\neq \pi_{0}.$ Hence $W=C_{W}(T)+[T,W].$

(2) Let $T^{'}\supset T$ be a maximal torus. According to (1), we
write $L=C_{L}(T^{'})+[T^{'},L].$ Since $[T^{'},L]\subseteq
[L,L]\subseteq I,$ it will
suffice to show that $C_{L}(T^{'})\subseteq T^{'}+I.$ Let $x\in
C_{L}(T^{'}).$ By virtue of Theorem \ref{t2.1.70}, there is $r$ such
that $x^{[p]^{r}}$ is  semisimple. As $x+I$ is a semisimple element
of $L/I,$ we find $n\geq r$ and $\alpha_{1},\cdots, \alpha_{n}\in
\mathbb{F}$ such that
$x-\sum\limits_{i=r}^{n}\alpha_{i}x^{[p]^{i}}\in I.$ Since
$\sum\limits_{i=r}^{n}\alpha_{i}x^{[p]^{i}}$ is a  semisimple
element of $C_{L}(T^{'})$ and $T^{'}$ is a maximal torus, we obtain
$\sum\limits_{i=r}^{n}\alpha_{i}x^{[p]^{i}}\in T^{'}.$ This
concludes our proof.
\end{proof}

\begin{theorem} \label{t2.3.7} Let $(L_{1},[p]_{1}),$  $(L_{2},[p]_{2})$  be restricted Leibniz algebras and $\varphi:L_{1}\rightarrow L_{2}$
a surjective $p$-homomorphism.

$\mathrm{(1)}$ If $T_{1}$ is a maximal torus of $L_{1},$ then $\varphi(T_{1})$ is a maximal torus of $L_{2}.$

$\mathrm{(2)}$ If $T_{2}$ is a maximal torus of $L_{2}$ and $T_{1}$ is a maximal torus of $\varphi^{-1}(T_{2}),$  then $T_{1}$ is a maximal torus of $L_{1}.$
\end{theorem}
\begin{proof} (1) Clearly, $\varphi(T_{1})$ is a torus of $L_{2}.$ Suppose that $T^{'}\supset \varphi(T_{1})$ is  a maximal torus of $L_{2}$. Then
 $\varphi^{-1}(T^{'})/\mathrm{ker}(\varphi)$ is a torus and by Lemma \ref{l2.3.6} (2) we may write $\varphi^{-1}(T^{'})=T_{1}+\mathrm{ker}(\varphi).$
 Hence $T^{'}=\varphi(\varphi^{-1}(T^{'}))=\varphi(T_{1}).$ This shows that $\varphi(T_{1})$ is a maximal torus of $L_{2}.$

(2) It follows from (1) that $\varphi(T_{1})=T_{2}.$  Let $T^{'}\supset T_{1}$ be a maximal torus of $L_{1}.$ Then $T_{2}\subseteq \varphi(T^{'})$
and the maximality of $T_{2}$ yields $\varphi(T^{'})=T_{2}.$ Thus $T^{'}\subseteq \varphi^{-1}(T_{2})$ and $T^{'}=T_{1},$ because of the maximality of $T_{1}.$
\end{proof}

\section{The uniqueness of decomposition}
Similar to Definition 2.1 of the reference \cite{czm}, we give the following definition.
\begin{definition} Let $\varphi$ be an endomorphism  of a restricted Leibniz algebra $(L,
  [p]).$ $\varphi$ is called an $L$-endomorphism of
   $L,$ if $\varphi L_{x}=L_{x} \varphi$ and $\varphi R_{x}=R_{x} \varphi$ for any
   $x\in L.$  An $L$-endomorphism of
   $L$ $\varphi$ is called an $L$-$p$-endomorphism of
   $L,$ if $\varphi(x^{[p]})=\varphi(x)^{[p]}, \forall x\in L.$ An $L$-endomorphism($L$-$p$-endomorphism) of
   $L$ $\varphi$ is called an $L$-automorphism($L$-$p$-automorphism)of
   $L,$ if $\varphi$ is bijection.
\end{definition}
\begin{example}  Let  $(L,[p])$ be a restricted  Leibniz algebra over
   $\mathbb{F}$ with   decomposition  $L=A\oplus B$ and $\pi$ be
   the projection into $A$ with respect to this decomposition,
    where $A$ and $B$ are $p$-ideals of $L$.
   Then $\pi$ is an $L$-$p$-endomorphism of $L$.
 \end{example}
\begin{lemma} \label{l3.2} Let  $(L,[p])$ be a restricted  Leibniz algebra over
   $\mathbb{F}.$ Then

$\mathrm{(1)}$ If $A$ is a subset of $L,$  then $Z_{L}(A)$ is a $p$-subalgebra of $L.$

$\mathrm{(2)}$ If $B$ is an ideal of $L,$  then $Z_{L}(B)$ is a $p$-ideal of $L.$ In particular, $Z(L)$ is a $p$-ideal of $L.$
\end{lemma}
\begin{proof} (1) For any $x,y\in Z_{L}(A), z\in A,$ we have $[[x,y],z]=[x,[y,z]]-[y,[x,z]]=0,$
$[x,y]\in Z_{L}(A).$ Similarly, $[y,x]\in Z_{L}(A).$
Hence $Z_{L}(A)$ is a subalgebra of $L.$  Since $x\in Z_{L}(A),$ one gets $[x^{[p]},z]=(L_{x})^{p-1}[x,z]=0,$
$x^{[p]}\in Z_{L}(A).$ As a result, $Z_{L}(A)$ is a $p$-subalgebra of $L.$

(2) For any $x\in Z_{L}(B), y\in L,z\in B,$ since $B$ is an ideal of $L,$  $[y,z]\in B,$ we have $[[x,y],z]=[x,[y,z]]-[y,[x,z]]=0,$
$[x,y]\in Z_{L}(B).$ Similarly, $[y,x]\in Z_{L}(B).$ Hence $Z_{L}(B)$ is an ideal of $L.$ By (1),
$Z_{L}(B)$ is a $p$-subalgebra of $L.$ Therefore, $Z_{L}(B)$ is a $p$-ideal of $L.$
\end{proof}

\begin{lemma} \label{l3.4} Let  $(L,[p])$ be a restricted Leibniz algebra
    over $\mathbb{F}$.    Then the following statements hold:

      $\mathrm{(1)}$ If $f$ and $g$ are  $L$-endomorphisms of $L$, then so are $f+g$ and $fg.$

     $\mathrm{(2)}$ If $f$ and $g$ are  $L$-$p$-endomorphisms of $L$, then so is $fg.$

     $\mathrm{(3)}$  If $f$ is an  $L$-$p$-automorphism of $L$, then  so is
     $f^{-1}$.
\end{lemma}
    \begin{proof} (1) Since $f$ and $g$ are  $L$-endomorphisms of
     $L,$  $fL_{x}=L_{x}f$ and $gL_{x}=L_{x}g$ for any
   $x\in L$. Then $(f+g)L_{x}=fL_{x}+gL_{x}=L_{x}f+L_{x}g=L_{x}(f+g),$ $(fg)L_{x}=f(gL_{x})=f(L_{x}g)=(fL_{x})g=(L_{x}f)g=L_{x}(fg).$
    Similarly, $(f+g)R_{x}=R_{x}(f+g), (fg)R_{x}=R_{x}(fg).$  So $f+g$ and $fg$ are $L$-endomorphisms of $L.$

   (2) Since $f$ and $g$ are  $L$-endomorphisms of
     $L$,  by (1), $fg$ is an $L$-endomorphism of $L.$ Clearly,
   $fg(x^{[p]})=(fg(x))^{[p]}.$ As a result, $fg$ is an $L$-$p$-endomorphism of $L.$

   (3) Since  $f$ is an  $L$-automorphism of $L,$ there is an
   automorphism $f^{-1}$ such that $f\cdot f^{-1}=f^{-1}\cdot
   f=\mathrm{id}_L$ and $fL_{x}=L_{x}f$ for any  $x\in L$. As  $(f\cdot f^{-1})L_{x}=L_{x}(f\cdot f^{-1}),$
   $f( f^{-1}L_{x})=L_{x}(f\cdot f^{-1})$ and $f^{-1}f( f^{-1}L_{x})=f^{-1}L_{x}(f\cdot f^{-1}),$ i.e.,
    $f^{-1}L_{x}=f^{-1}L_{x}(f\cdot f^{-1})=f^{-1}(fL_{x}) f^{-1}=L_{x} f^{-1}.$ Similarly,
    $f^{-1}R_{x}=R_{x}f^{-1}.$ So $f^{-1}$ is an $L$-automorphism of $L.$ $x^{[p]}=(ff^{-1}(x))^{[p]}=f((f^{-1}(x))^{[p]}),
     \forall x\in L,$ $f^{-1}(x^{[p]})=(f^{-1}(x))^{[p]}.$
   Hence $f^{-1}$ is an  $L$-$p$-automorphism of $L.$
   \end{proof}
\begin{lemma} \label{l3.5}  Let $(L,[p])$ be a restricted Leibniz algebra over $\mathbb{F}.$ If $\varphi$ is an $L$-$p$-endomorphi-\\sm of $L$,
  then there exists  $k\in \mathbb{N}\backslash\{0\}$ satisfying

  $\mathrm{(1)}$  $L$ has a decomposition of $p$-ideals
  $L=\mathrm{ker}\varphi^k \oplus \mathrm{Im}\varphi^k.$

  $\mathrm{(2)}$  If $L$ can not
   be decomposed into the direct sum of  $p$-ideals of $L,$ then $\varphi^k=0$ or
   $\varphi\in \mathrm{Aut}_{p}L,$ where $\mathrm{Aut}_{p}L$ is the group of $p$-automorphisms of $L.$
\end{lemma}
\begin{proof} (1) Let $f(\lambda)=\lambda^kg(\lambda)$ be the
   minimal polynomial of $\varphi,$
   where $\lambda$ and $g(\lambda)$ are coprime. Then there are polynomials
   $u(\lambda)$ and $ v(\lambda)$ satisfying
  $u(\lambda)g(\lambda)+v(\lambda)\lambda^k=1.$ So we have
     $y=u(\varphi)g(\varphi)(y)+v(\varphi)\varphi^k(y)$ for all $y\in
     L.$ Since $\varphi^k(u(\varphi)g(\varphi)(y))=(\varphi^kg(\varphi))u(\varphi)(y)=0,
      u(\varphi)g(\varphi)(y)\in \ker \varphi^k,$
    and $v(\varphi)\varphi^k(y)=\varphi^k (v(\varphi)(y))\in {\rm Im}\varphi^k.$ Thus $L=\ker \varphi^k
    + {\rm Im}\varphi^k.$ If $y\in \ker \varphi^k
    \cap {\rm Im}\varphi^k,$ then $\varphi^k(y)=0$ and
    $y=\varphi^k(z)$ for some $z\in L.$ So
    $y=u(\varphi)g(\varphi)\varphi^k(z)
     +v(\varphi)\varphi^k(y)=u(\varphi)f(\varphi)(z)
     +v(\varphi)(\varphi^k(y))=0,$ i.e., $\ker \varphi^k
    \cap {\rm Im}\varphi^k=\{0\}.$ Thus $L=\ker \varphi^k
    \oplus {\rm Im}\varphi^k$ as a vector space.

     Since $\varphi$ is an $L$-$p$-endomorphism of
       $L,$ $\varphi^k$ is an $L$-$p$-endomorphism
       of $L$  by Lemma
       \ref{l3.4} (2). Then $\varphi^k[x,L]=[\varphi^k(x),\varphi^k(L)]=0,  \varphi^k[L,x]=[\varphi^k(L),\varphi^k(x)]=0$ for any
       $x\in \ker\varphi^k,$ i.e.,  $\ker\varphi^k$ is an ideal of
       $L.$ $\varphi^k(x^{[p]})=\varphi^{k-1}(\varphi(x)^{[p]})=(\varphi^{k}(x))^{[p]}=0, \forall x\in \ker\varphi^k.$
      Hence $\ker\varphi^k$ is a $p$-ideal of $L.$
      Let $x=x_{1}+\varphi^{k}(x_{2})\in
       L,$ where $x_{1}\in \ker\varphi^k, x_{2}\in L.$ Suppose $a\in {\rm Im}\varphi^k.$ So $a=\varphi^k(y)$ for
       some $y\in L.$ Since $\varphi^k$ is an $L$-$p$-endomorphism of
       $L,$  then  $[x,a]=[x,\varphi^k(y)]=\varphi^k[x_{2},y]\in {\rm Im}\varphi^k.$ Similarly,  $[a,x]\in {\rm Im}\varphi^k.$
       Therefore,  ${\rm Im}\varphi^k$ is an ideal of $L.$ Let $x\in {\rm Im}\varphi^k.$ Then there exists $y\in L$ such that $x=\varphi^k(y).$ $x^{[p]}=(\varphi^k(y))^{[p]}=\varphi^k(y^{[p]}),$ $x^{[p]}\in {\rm Im}\varphi^k.$ Consequently,
 ${\rm Im}\varphi^k$ is a $p$-ideal of $L.$

  (2) If $L$ can not
   be decomposed into the direct sum of  $p$-ideals, then we can know that $\ker
  \varphi^k=L$ or Im$\varphi^k=L$ by (1). This means that $\varphi^k=0$ or
  $\varphi^k \in {\rm Aut}_{p}L.$ So $\varphi^k=0$ or
   $\varphi\in \mathrm{Aut}_{p}L.$
\end{proof}
 \begin{lemma} \label{l3.6} Let $(L,[p])$ be a restricted  Leibniz algebra over $\mathbb{F}.$ Let $\varphi_i(1\leq i\leq n),$
  $\sum\limits_{i=1}^j \varphi_i(1\leq j\leq n)$ be
  $L$-$p$-endomorphisms of $L$ and $\varphi_1+\varphi_2 +\cdots
  +\varphi_n=\mathrm{id}_L.$ If $L$ can not
   be decomposed into the direct sum of  $p$-ideals,  then there exists $i( 1\leq i \leq n)$ satisfying  $\varphi_i\in
  {{\rm Aut}_{p}}L.$
  \end{lemma}
\begin{proof} We prove this result by induction on $n.$
   The result is obviously true for $n=1.$
For $n=2,$ since $\varphi_1+\varphi_2=\mathrm{id}_L,$
      $\varphi_1(\varphi_1+\varphi_2)=(\varphi_1+\varphi_2)\varphi_1$
  and $\varphi_1\varphi_2=\varphi_2\varphi_1.$ Now, we suppose
  $\varphi_1,\varphi_2\notin {{\rm Aut}_{p}}L.$ By virtue of Lemma  \ref{l3.5} (2), there is
  $k_i(i=1,2)$ satisfying $\varphi^{k_i}=0.$ Put $k>k_1+k_2,$
  then
 $\mathrm{id}_L=(\varphi_1+\varphi_2)^k=\sum_{j=0}^kC_k^j\varphi_1^{k-j}\varphi_2^j=0.$
 It is a contradiction.  From it we can get
  $\varphi_1\in{\rm Aut}_{p} L$ or
     $\varphi_2\in{{\rm Aut}_{p}}L.$

  Suppose $n-1$ holds and $\psi:=\sum\limits_{i=1}^{n-1}\varphi_i,$ then
  $\psi+\varphi_n=\mathrm{id}_L.$
    From the discussion in the case of $n=2,$
    we get $\psi\in{{\rm Aut}_{p}}L$ or $\varphi_n\in{{\rm Aut}_{p}}L.$ If
   $\varphi_n\in{{\rm Aut}_{p}}L,$ then the conclusion is true. If
  $\psi\in{{\rm Aut}_{p}}L,$ then $\psi^{-1},$  $\varphi_1\psi^{-1},$
  $\cdots,$ $\varphi_{n-1}\psi^{-1}$ are
  $L$-$p$-endomorphisms of $L$ by means of Lemma \ref{l3.4} and
   $\sum\limits_{i=1}^{n-1}\varphi_i\psi^{-1}=\psi\psi^{-1}=\mathrm{id}_L$.
   By the inductive assumption, there exists $i$ such that $\varphi_i\psi^{-1}\in{\rm
   Aut}_{p}L.$ Hence $\varphi_i\in{{\rm Aut}_{p}}L.$
\end{proof}
   \begin{lemma} \label{l3.7}  Let  $(L,[p])$ be a restricted  Leibniz algebra over $\mathbb{F}.$
    If $L$ has a decomposition of $p$-ideals $L=A\oplus B,$ then the following statements hold:

       $\mathrm{(1)}$ $Z(L)$ has a decomposition of $p$-ideals $Z(L)=Z(A)\oplus Z(B).$

       $\mathrm{(2)}$ If $Z(L)=0,$ then $Z_{L}(A)=B$ and $ Z_{L}(B)=A.$
       \end{lemma}
  \begin{proof} (1) According to Lemma \ref{l3.2}, $Z(A)$ and $Z(B)$ are $p$-ideals of
    $L$. Since $Z(A) \cap Z(B)=\{0\},$ we have $Z(A) \oplus Z(B) \subseteq Z(L)$. Now,
     suppose $x\in Z(L)$ and $x=x_{1}+x_{2},$ where $x_{1}\in A, x_{2}\in B.$ Then $[x_{1},A]=[x-x_{2},A]=0.$
     Hence $x_{1}\in Z(A).$ Similarly, $x_{2}\in Z(B).$  Hence $Z(L)=Z(A)\oplus Z(B).$

   (2) $B\subseteq Z_{L}(A)$ is obviously true. It is sufficient to show that $Z_{L}(A)\subseteq B.$
     Since $L=A\oplus B,$ for any element $x$ of $Z_{L}(A),$ we have  $x=x_{1}+x_{2},$ where $x_{1}\in A, x_{2}\in B.$
      It follows that $$0=[x,a]=[x_{1}+x_{2},a]=[x_{1},a]+[x_{2},a]=[x_{1},a]$$
     for all $a\in A.$ Thus $x_{1}\in Z(A)=0.$ Hence $x=x_{1}+x_{2}=x_{2}\in B$ and $Z_{L}(A)\subseteq B.$
     Consequently, $Z_{L}(A)=B.$ Similarly, we can get that $Z_{L}(B)=A.$
\end{proof}
    \begin{lemma} \label{l3.8}  Let  $(L,[p])$ be a restricted  Leibniz algebra
    over $\mathbb{F}$
    such that $L=A\oplus B$. If $A$ and $B$ are  ideals of
    $L$ and $C$ is a  subalgebra of $L$ such that
    $A\subseteq C$, then $C=A\oplus (C\cap B)$ and
     $C\lhd L$ if and only if $(B\cap C)\lhd B$.
     \end{lemma}

  \begin{proof} Since $B$ is an  ideal and $C$ is a  subalgebra of $L$,
     $[C\cap B, C]\subseteq [B,
   C]\subseteq B$ and $[C\cap B, C]\subseteq [C,
   C]\subseteq C$. Then $[C\cap B, C]\subseteq C\cap B$,
   i.e., $C\cap B$ is an ideal of $C$.  So
   there is an isomorphism such that $(B+C)/B\cong
   C/C\cap B$. On the other hand, $(C+B)/B\cong A$. Hence $A\cong
    C/C\cap B$ and $C=A\oplus (C\cap B)$.
    The second statement is clear.
\end{proof}

 \begin{theorem} \label{t3.9}   Suppose that  a  restricted Leibniz algebra
   $(L,[p])$ over $\mathbb{F}$  has decompositions of $p$-ideals
\begin{equation}L=M_1\oplus M_2\oplus \cdots \oplus
      M_s,
 \label{e4.1.60a}
\end{equation}
  \begin{equation}L=N_1\oplus  N_2\oplus \cdots \oplus
      N_t,
 \label{e4.1.60b}
 \end{equation}
where $M_1,\cdots, M_s$ and $N_1,\cdots, N_t$ can not be decomposed into the direct sum of $p$-ideals.
If $Z(L)=0,$ then $s=t$ and $M_i=N_i,$ $i=1,2,\cdots, s$ after changing the
     orders.
\end{theorem}
\begin{proof} We prove this theorem by induction on $n$.
   If $s=1$, then  $L$ can not be decomposed into the direct sum of  $p$-ideals. So $t=1$
  and $M_1=N_1=L$.

  Now put $s>1$, naturally $t>1$, too.
   Let $\pi$ be the projection of $L$ to $M_1$ with respect to the decomposition  $(\ref{e4.1.60a}),$
   $\sigma$ the imbedding of $M_1$
   to $L,$  $\rho_i$ the projection of $L$ to $N_i$  with respect to the
     decomposition $(\ref{e4.1.60b})$ and $\tau_i$  the imbedding of
    $N_i$ to $L$. Then  $\pi,
   \rho_1, \cdots, \rho_t$ and $\sum\limits_{i=1}^k\rho_i (1\le k\le
    t)$ are $L$-$p$-endomorphisms of  $L$ and
   $\rho_1+\rho_2+\cdots +\rho_t={\rm id}_{L}.$
   Letting $ \pi_i^* =\pi \tau_i=\pi|_{N_i}, \rho_i^*=\rho_i\sigma=\rho_i|_{ M_1}$ for
   any $i=1,2,\cdots,t,$  then  $\pi_i^*\rho_i^*$ is the $M_1$-$p$-endomorphism of $M_1$.

   Defined $\sum\limits_{i=1}^j\tau_i\rho_i:L\rightarrow L$ by
   $(\sum\limits_{i=1}^j\tau_i\rho_i)(x)=\sum\limits_{i=1}^j\tau_i\rho_i(x)$ for all $x\in L, 1\leq j\leq t$.
   We verify that the mapping is an  $L$-$p$-endomorphism of  $L$.
   In fact, for $x,y\in L,$ we write $x=\sum\limits_{i=1}^{t}x_{i}, y=\sum\limits_{i=1}^{t}y_{i},$
   where $x_{i},y_{i}\in B_{i}(1\leq i\leq t).$ So $$\sum_{i=1}^{j}\tau_{i}\rho_{i}(x^{[p]})=
   \sum_{i=1}^{j}\tau_{i}(\rho_{i}(x)^{[p]})=\sum_{i=1}^{j}\tau_{i}(x_{i}^{[p]})=\sum_{i=1}^{j}x_{i}^{[p]}.$$
   On the other hand, from $L=N_1\oplus  N_2\oplus \cdots \oplus N_t,$ we can obtain that
   $[N_{i}, N_{j}]=0$ for $1\leq i,j\leq t$ with $i\neq j.$ Hence we may imply that
   $$(\sum_{i=1}^{j}\tau_{i}\rho_{i}(x))^{[p]}=
   (\sum_{i=1}^{j}x_{i})^{[p]}=\sum_{i=1}^{j}x_{i}^{[p]}$$ and
   $\sum\limits_{i=1}^{j}\tau_{i}\rho_{i}(x^{[p]})=(\sum\limits_{i=1}^{j}\tau_{i}\rho_{i}(x))^{[p]}.$
  Next we show that $\sum\limits_{i=1}^{j}\tau_{i}\rho_{i}$ is an endomorphism of  $L.$
  By virtue of $[N_{i},N_{j}]=0,$ we have
   \begin{equation*}
\begin{split}
&\sum_{i=1}^{j}\tau_{i}\rho_{i}[x,y]=\sum_{i=1}^{j}\tau_{i}[\rho_{i}(x),\rho_{i}(y)]
=\sum_{i=1}^{j}[x_{i},y_{i}]
=[\sum_{i=1}^{j}x_{i},\sum_{i=1}^{j}y_{i}]
=[\sum_{i=1}^{j}\tau_{i}\rho_{i}(x),\sum_{i=1}^{j}\tau_{i}\rho_{i}(y)].
\end{split}
\end{equation*}
Finally, using similar method we may verify that
$$(\sum\limits_{i=1}^{j}\tau_{i}\rho_{i})L_{x}=L_{x}(\sum\limits_{i=1}^{j}\tau_{i}\rho_{i})
\quad {\rm and}\quad
(\sum\limits_{i=1}^{j}\tau_{i}\rho_{i})R_{x}=R_{x}(\sum\limits_{i=1}^{j}\tau_{i}\rho_{i}).$$
Thus $\sum\limits_{i=1}^{j}\tau_{i}\rho_{i}$ is an
$L$-$p$-endomorphism of $L$.
    Furthermore,
   $\pi(\sum\limits_{i=1}^j\tau_i\rho_i)\sigma=
  \sum\limits_{i=1}^j\pi_i^*\rho_i^*=\sum\limits_{i=1}^j\pi_i^*\rho_i|_{ M_1}$ is an
   $ M_1$-endomorphism of $M_1$.
    For each $h\in  M_1$, we have
    $h=\pi(h)=\pi(\sum\limits_{i=1}^t\rho_i(h))=\sum\limits_{i=1}^t\pi_i^*\rho_i^*(h),$
    then $\sum\limits_{i=1}^t\pi_i^*\rho_i^*={\rm id}_{M_1}.$ So
     there exists an index $i$ satisfying $\pi_i^*\rho_i^*\in{\rm
   Aut}_{p}M_1$ by virtue of Lemma \ref{l3.6}. If needed, after changing the order of $N_1,$ $N_2,$
    $\cdots,$ $ N_t,$ we can get $i=1,$ $\pi_1^*\rho_1^*\in{\rm
    Aut}_{p}M_1.$ Thus $\rho_1^*$ is a bijection. Let
     $M=M_2\oplus  M_3 \oplus\cdots\oplus  M_s,\, N=N_2\oplus
     N_3 \oplus\cdots\oplus  N_t.$  By Lemma \ref{l3.7}, we have
   $Z(M)=Z(N)=\{0\}$ and $M=Z_L(M_1),  M_1=Z_L(M),$ $\ker
  \rho_1=N=Z_L(N_1),  N_1=Z_L(N).$ Hence
   $\{0\}=\ker \rho_1^*=M_1\cap \ker\rho_1=M_1\cap  N.$
  So we have $M_1\subseteq  Z_L(N)=N_1,$
  $N_1=M_1\oplus (N_1 \cap M)$ by Lemma \ref{l3.8}.
   But $N_1$ can not be decomposed into the direct sum of  $p$-ideals, then $N_1=M_1.$ By inductive assumption
  we obtain the desired result.
\end{proof}
\begin{corollary}
Let $(L,[p])$  be a restricted Lie algebra
    over $\mathbb{F}$ with trivial center. If $L$ has a decomposition of $p$-ideals
$L=A_1\oplus A_2\oplus \cdots \oplus A_s$, then the decomposition is unique after changing the orders,
where $A_1,\cdots, A_s$ can not be decomposed into the direct sum of  $p$-ideals.
\end{corollary}

\end{document}